\documentclass[11pt]{article}
\usepackage[margin=0.1in,footskip=0.2in]{geometry}
\usepackage[english]{babel}
\usepackage{graphicx}
\usepackage{amssymb}
\usepackage{amsthm}
\usepackage{amssymb}
\usepackage{setspace}

\usepackage{amsmath,amsthm,amscd,amsfonts,mathrsfs,amssymb}
\usepackage{graphicx,enumerate}

\usepackage[normalem]{ulem}
\usepackage[all]{xy}
\usepackage{fullpage} 
\usepackage{color}
\usepackage{makeidx}
\usepackage{alltt}
\usepackage{setspace}
\usepackage{hyperref}

\newtheorem{theorem}{Theorem}[section]
\newtheorem{lemma}[theorem]{Lemma}

\theoremstyle{definition}
\newtheorem{remark}[theorem]{Remark}
\theoremstyle{definition}

\DeclareMathOperator{\dd}{{\mathrm{d}}}
\newcommand{\ssl}[2]{\mathrm{GL}(#1,\mathbb{#2})}
\newcommand{\ggl}[2]{\mathrm{GL}(#1,\mathbb{#2})}
\DeclareMathOperator{\GL}{GL}

\newcommand{\Z}{ \mathbb{Z}}
\newcommand{\C}{\mathbb{C}}

\newcommand{\bra}[1]{\left(#1\right)}
\newcommand{\sign}{\operatorname{sgn}}

\newcommand{\blue}[1]{\textcolor{blue}{#1}}

\newcommand{\cyan}[1]{\textcolor{cyan}{#1}}

\newcommand{\nothing}[1]{}

\newcommand\srel[2]{\begin{smallmatrix} {#1} \\ {#2} \end{smallmatrix}}

\DeclareMathOperator{\Kl}{\mathrm{Kl}}
\DeclareMathOperator{\gama}{{\mathbf{\mathsf G}}}
\newcommand{\sums}{\sideset{}{^*}\sum}
\renewcommand{\vec}{\mathbf}
\renewcommand{\mod}{\textrm{mod }}

\newcommand{\cc}{\tfrac{Q_1\cdots Q_{M} c}{D_1 \cdots D_{M} }}
\newcommand{\cb}{\tfrac{Q_1\cdots Q_{M} c}{b_1 \cdots b_{M} }}

\newcommand{\sumD}{\sum\limits_{\vec{D}|\vec{Q}}} 
\newcommand{\sumdd}{\sum\limits_{\vec{d}|\vec{q}}} 

\begin{document}

\title{\scshape The balanced Voronoi formulas for $\textrm{GL}(n)$}

\author{Stephen D. Miller\thanks{Supported by NSF grant DMS-1500562}{ }  and Fan Zhou}

\maketitle
\begin{abstract}
In this paper we show how the $\textrm{GL}(N)$ Voronoi summation formula of \cite{millerschmid2} can be rewritten to incorporate hyper-Kloosterman sums of various 
dimensions  on both sides.  This generalizes a formula for $\textrm{GL}(4)$ with ordinary Kloosterman sums on both sides that was used in \cite{BLM} to prove nonvanishing of $GL(4)$ $L$-functions by $GL(2)$-twists, and later by the second-named author in \cite{zhou}.
\\
\\
\end{abstract}
\section{Introduction}\label{sec:intro}

The Voronoi summation formula for GL(2) has long been a standard tool for studying analytic properties of automorphic forms and their $L$-functions.
More recently, the Voronoi formula for GL(3) of the first-named author and Schmid in \cite{millerschmid1} has found applications in the study of automorphic forms on GL(3) and their $L$-functions, such as \cite{miller}, \cite{li}, \cite{munshi}, and \cite{fouvry}.
The Voronoi formula was generalized to GL($N$) in \cite{millerschmid2}, with other proofs later
 found by   \cite{goldfeldli1}, \cite{goldfeldli2}, \cite{ichinotemplier}, and \cite{kiralzhou}.

The existing Voronoi formula  for GL($N$), $N\geq 3$, (e.g., Theorem~\ref{thm:originalvoronoi}) is a Poisson-style summation formula with Fourier coefficients of an automorphic form twisted by additive characters	 on one side, and those of a contragredient form twisted by (hyper-)Kloosterman sums
of dimension $N-2$ on the other side.
The appearance of the (hyper-)Kloosterman sums was already suggested by finite harmonic analysis with Dirichlet characters and Gauss sums, e.g.,  in \cite{dukeiwaniec}.

In 2011, the first-named author and Xiaoqing Li discovered a different (so called ``balanced'')  Voronoi-type formula on GL(4), with both sides twisted by ordinary Kloosterman sums (see \cite{BLM} and \cite[Theorem 1.2]{zhou}).
This formula was first derived by modifying the  automorphic-distributional proof  in \cite{millerschmid2}.
The second-named author later generalized that formula to GL($N$)  under certain hypotheses
(\cite[Theorem 1.1]{zhou}). In this paper, we complete the general balanced Voronoi formulas for cusp forms on $\ggl N Z\backslash \ggl N R$.
These balanced formulas are derived from the original Voronoi formula of \cite{millerschmid2}, and equate  a  sum of Fourier coefficients twisted by hyper-Kloosterman sums of dimension $L$  with a contragredient sum  twisted by hyper-Kloosterman sums of dimension $M$, where  $N=L+M+2$.	
The original formula of Miller and Schmid corresponds to the case of $L=0$ and $M=N-2$, while
the balanced formula of Li and Miller on GL(4) corresponds to the case of $L=1$ and $M=1$. 
The latter formula on GL(4) is a key ingredient in the recent nonvanishing theorem for GL(2)-twists of GL(4) $L$-functions in \cite{BLM}.  This is because the Kloosterman sums in the balanced Voronoi formula on GL(4) mesh well with the Kloosterman sums appearing in the Kuznetsov trace formula on GL(2).  The match between them is used in \cite{BLM} to create a  spectral reciprocity formula, from which mean value estimates and the nonvanishing result are deduced.

The proof of our balanced formulas (Theorem \ref{thm:balancedvoronoi}) in this paper is different from the automorphic-distributional method used to prove Li and Miller's balanced formula on GL(4). Our proof is also different from that of \cite[Theorem 1.1]{zhou}, which instead uses functional equations of twisted automorphic $L$-functions.

\nothing{
Let $$S(a,b;c):=\sum_{\substack{r\;\mod c\\(r,c)=1}}\exp\bra{2\pi i\frac{ar+b\bar r}{c}}$$
be the ordinary Kloosterman sum.
}


Before stating the formulas, we need define the hyper-Kloosterman sums which already appear in the Voronoi formula of \cite{millerschmid2} for GL($N$) for $N\geq 4$ (restated in Theorem~\ref{thm:originalvoronoi} below).
Denote $e(x):=\exp(2\pi i x)$.
Let $a,n \in \Z$,  $c\in \mathbb{N}$, and let
	$$
		\vec{q}=  (q_1,q_2, \ldots, q_{\mathsf{N}}) \qquad \text{ and } \qquad
		\vec{d} = (d_1, d_2, \ldots, d_{\mathsf{N}})
	$$
be $\mathsf{N}$-tuples of positive integers satisfying the divisibility conditions
\begin{equation}\label{eq:divisibilityConditions}
d_1|q_1c\,,\quad d_2\left|\frac{q_1q_2c}{d_1}\,,\qquad \cdots,\qquad
d_{\mathsf{N}}\right|\frac{q_1\cdots   q_{\mathsf{N}}c}{d_1 \cdots d_{\mathsf{N}-1}}\,.
\end{equation}
Define the $\mathsf{N}$-dimensional hyper-Kloosterman sum as
	\begin{equation}\nonumber
		\Kl_{\mathsf{N}}(a,n,c;\vec{q},\vec{d})  \ \  :=    \ \
        \sums_{\srel{\srel{x_1 (\text{mod } \frac{q_1c}{d_1})}{x_2 (\text{mod } \frac{q_1q_2c}{d_1d_2})}}{\srel{\vdots}{x_{\mathsf{N}} (\text{mod } \frac{q_1\cdots q_{\mathsf{N}} c}{d_1 \cdots d_{\mathsf{N}} })}}}
		e\left(\frac{d_1x_1a}{c} + \frac{d_2x_2 \overline{x_1}}{\frac{q_1c}{d_1}}
		+ \cdots +
		\frac{d_{\mathsf{N}}x_{\mathsf{N}} \overline{x_{\mathsf{N}-1}}}{\frac{q_1\cdots q_{\mathsf{N}-1} c}{d_1 \cdots d_{\mathsf{N}-1} }}
		+ \frac{n \overline{x_{\mathsf{N}}}}{\frac{q_1\cdots q_{\mathsf{N}} c}{d_1 \cdots d_{\mathsf{N}} }}\right),
\label{eq:hyper-kloosterman}
	\end{equation}
	where $\sums$ indicates that the summations are restricted to coprime residue classes  and
	$\overline{x_i}$ denotes the multiplicative inverse of $x_i$  modulo  $\frac{q_1\cdots q_i c}{d_1 \cdots d_i}$. In the degenerate case of $\mathsf{N}=0$, we define
$\Kl_0(a,n,c; \;, \;)=e\bra{\frac{an}{c}}$; when $\mathsf{N} = 1$ the hyper-Kloosterman sum $\Kl_1(a,n,c; q_1 ,d_1 )$ reduces to the ordinary Kloosterman sum $S(aq_ 1 ,n;q_ 1c /d_ 1 )$.


Let $F$ be a cuspidal automorphic form for $\ggl N Z$.  As is customary, we assume that $F$ generates an irreducible subrepresentation $\pi$ of $L^2_\xi(Z_{\mathbb{R}}\ggl N Z \backslash \ggl N R)$ under the right regular representation of $\ggl N R$, where $Z_{\mathbb{R}}$ denotes the center of $\ggl N R$ and $\xi$ is a central character.  Note this does not imply that   $F$ is a Hecke eigenform, which is a stronger assumption that is unnecessary using our methods.
Let $A(*,\cdots,*)$ denote its abelian Fourier coefficients (see \cite[(2.9)]{millerschmid2} and \cite[(2.1.5)]{bump}),
which are the Hecke eigenvalues of $F$ when $F$ is a normalized Hecke eigenform.
The Voronoi summation formula in \cite{millerschmid2} is a Poisson-sum style identity relating sums of the abelian Fourier coefficients weighted against   test functions $\omega$ and $\Omega$, which are related by an integral transform completely  determined by $\pi$.  Further background on Voronoi summation    and this integral transform (which our new formula shares as well) is given in Section~\ref{sec:formerappendix}.

There are various ways to describe  allowable choices of test functions $\omega$ in the Voronoi summation formula.  The   simplest approach (which we follow here) is to demand that $\omega$ be a smooth function on $\mathbb R$ which has compact support contained in   $ \mathbb{R}_{>0} =(0,\infty)$; this is natural since $\omega(x)$ is never evaluated at  $x=0$ in  Theorem~\ref{thm:mainthm_testfunction}.
However, for some applications (e.g., to $L$-functions) it is important to allow different behavior at the origin,  such as fractional powers of the form $|x|^s$ or $|x|^s\sign(x)$ for $s\in \C$.  We shall not pursue this here, other than noting that any admissible function used in the usual Voronoi formula on $\GL(N)$ (see \cite[(1.8)]{millerschmid2}) can be used in the balanced Voronoi formulas (with only minor modifications to account for parities); this is because our proof constructs the  balanced formula  as a finite average of formulas of the type given in Theorem~\ref{thm:originalvoronoi}.
 At a formal level, the integral  transform  has the form
\begin{equation}\label{formalvoronoitransform}
\Omega(y) \ \ = \ \ \frac{1}{|y|}\int_{\mathbb{R}^N} \omega\bra{\frac{x_1\cdots x_N}{y}} \prod_{1\leq j\leq N} \bra{e(-x_j)|x_j|^{-\lambda_j}\sign (x_j)^{\delta_j}\dd \!x_j},
\end{equation}
where the $\lambda_j$ and $\delta_j$ are the representation parameters of $\pi$ (this notion as well as a reformulation of (\ref{formalvoronoitransform}) in terms of Mellin inversion is given in  Section~\ref{sec:formerappendix}; see also \cite[\S1]{millerschmid2}).

\begin{theorem}\label{thm:mainthm_testfunction}
Let $F$ be a cuspidal automorphic form on $\ssl{N}{Z}\backslash \ggl N R$  for $N\geq 3$ with abelian Fourier coefficients $A(*,\ldots, *)$, and which generates an irreducible representation of $\ggl N R$.
Let $\omega\in C^\infty_{c}({\mathbb R}_{>0})$ and let $L$ and $M$ be two non-negative integers with $L+M+2=N$.  Let $c>0$ be an integer and let $a$ be any integer with $(a,c)=1$.
Denote by $\overline{a}$ the multiplicative inverse of $a$  modulo $c$.
Let $\vec{q}=(q_1,q_2,\cdots,q_L)$ be an $L$-tuple of positive integers and let
$\vec{Q}=(Q_1,Q_2,\cdots,Q_M)$ be an $M$-tuple of positive integers.
Let $\sumD$ stand for $ \sum_{D_1|Q_1c}\sum_{D_2 | \frac{Q_1Q_2c}{D_1}}
		\cdots \sum_{D_{M}| \frac{Q_1\ldots Q_M c}{D_1 \ldots D_{M-1}}}$
and let $\sumdd$ stand for 		$ \sum_{d_1|q_1c}\sum_{d_2 | \frac{q_1q_2c}{d_1}}
		\cdots \sum_{d_{L}| \frac{q_1\ldots q_L c}{d_1 \ldots d_{L-1}}}$. Then
	\begin{align*}
	\begin{split}
	\sumD &\sum_{n=1}^\infty A(q_{L},\cdots, q_1,D_1,\cdots, D_M ,n)\Kl_M(\bar{a},n, c; \vec{Q}, \vec{D})
D_1^MD_2^{M-1}\cdots D_M\; \omega\!\bra{\frac{nD_1^{M+1}D_2^M\cdots D_M^2}{q_1^Lq_2^{L-1}\cdots q_L}}\\
&= \sumdd 				d_1^Ld_2^{L-1}\cdots d_L
		 \sum_{n=1}^\infty
		\frac{A(n,d_{L}, \ldots,  d_1, Q_1,\cdots, Q_M) \Kl_L(a,n, c; \vec{q}, \vec{d}) }
		{c^{L+1}}
\Omega\!\bra{\frac{(-1)^ { M+1} nd_1^{L+1}d_2^L\cdots d_L^2}{c^NQ_1^MQ_2^{M-1}\cdots Q_M}}		
				\\
		&+  \sumdd				d_1^Ld_2^{L-1}\cdots d_L \sum_{n=1}^\infty
		\frac{A(n,d_{L}, \ldots,  d_1, Q_1,\cdots, Q_M) \Kl_L(a,-n, c; \vec{q}, \vec{d}) }
		{c^{L+1}}
\Omega\!\bra{\frac{(-1)^{M}nd_1^{L+1}d_2^L\cdots d_L^2}{c^NQ_1^MQ_2^{M-1}\cdots Q_M}},		
		\end{split}
	\end{align*}
where $\Omega$ is the integral  transform from (\ref{formalvoronoitransform}) (which is rigorously defined as a convergent integral in (\ref{eq:Omega+-def_new})-(\ref{eq:transform_integration_w_gamma})).
\end{theorem}

\begin{remark}
\nothing{To prove of Theorem \ref{thm:mainthm_testfunction} is to prove
Theorem \ref{thm:balancedvoronoi}.
}
As we mentioned earlier,  Theorem~\ref{thm:mainthm_testfunction} is proved by averaging  over  a finite number of instances of the original Voronoi formula of \cite{millerschmid2} (Theorem \ref{thm:originalvoronoi}).  Consequently, any analysis of test functions for that formula automatically transfers over to our present setting.
The construction by finite average also shows that any coefficients $A(*,\ldots,*)$ satisfying the summation formula in \cite{millerschmid2} must also satisfy the summation formula in Theorem~\ref{thm:balancedvoronoi}.  This extends the range of applicability of $F$ to cases where functoriality has not yet been shown.
 For example,  K{\i}ral and the second-named author have shown in \cite{kiralzhou} that the Voronoi summation formula of \cite{millerschmid2} also holds when $F$ is a Rankin-Selberg convolution of two full-level cuspidal automorphic representations (see \cite[Examples 1.8-1.9]{kiralzhou}). Therefore Theorem  \ref{thm:mainthm_testfunction} and Theorem \ref{thm:balancedvoronoi} hold for such $F$ as well, despite it not yet being known to be automorphic.
\end{remark}

\begin{remark}
We have stated the summation formula in Theorem~\ref{thm:mainthm_testfunction} so that it only involves a sum over positive integers $n$ on the lefthand side.  This is somewhat unnatural from the point of view of automorphic distributions, through which one obtains summation formulas via integration against   distributions that involve  terms for both positive and negative $n$.  Also, including both positive and negative $n$ on the lefthand side results in simplifying the righthand side, as well as the analytic assumptions on the behavior of $\omega$ near the origin.   Nevertheless, since Voronoi summation formulas are typically applied to sums indexed by positive integers $n$, we have chosen to sacrifice aesthetics for practicality and state our formula as above.
\end{remark}















\section{Voronoi formulas as Dirichlet series identities}\label{sec:formerappendix}

Let $F$ be a cuspidal automorphic form on  $\ssl{N}{Z}\backslash \ssl N R$ and let $\pi$ denote the archimedean representation attached to $F$, which we assume is irreducible.
We say that $(\lambda,\delta)\in \mathbb{C}^N\times (\mathbb Z/2\mathbb Z)^N$  is a representation parameter of  $F$  if $\pi$ embeds into a subspace of the principal series representation
$$V_{\lambda,\delta} \ \  = \ \  \left\{f:GL(N,\mathbb R)\to \mathbb C\,\left|\,f\bra{g\begin{pmatrix}
a_1 &0 &0\\
* &\ddots &0 \\
* &* &a_n
\end{pmatrix}} = f(g) \prod_{j\le N} \bra{|a_j|^{\frac{N+1}{2}-j-\lambda_j}\sign(a_j)^{\delta_j}}\right. \right\}, $$
which is a representation space for $\ggl N R$ under the  left translation action $[\pi_{\lambda,\delta}(g)f]h=f(g^{-1}h)$.  When $\pi$ is spherical, any simultaneous permutation of the entries of $(\lambda,\delta)$ is also a representation parameter; in this case $\lambda$ coincides with the notion  of Langlands parameter, though it does not in general (see \cite[A.1-A.2]{millerschmid3} for a complete description of all allowable representation parameters of cuspidal automorphic representations of $\ggl N R$).  The $\ssl{N}{Z}$-invariance forces $\delta_1+\cdots+\delta_n\equiv 0~(\mod 2)$ \cite[(2.2)]{millerschmid2}.

Define the Gamma factor
$$G_\delta(s) := \begin{cases}
2 (2\pi)^{-s}\Gamma(s) \cos (\pi s/2)\,, \quad &\text{ if }\delta\,\in\,2\Z,\\
2i (2\pi)^{-s}\Gamma(s) \sin (\pi s/2)\,, \quad &\text{ if }\delta\,\in\,2\Z+1.
\end{cases}
$$
Alternatively, if $\Gamma_{\mathbb R}(s) = \pi^{-s/2}\Gamma(s/2)$ denotes the usual Artin Gamma factor appearing in the functional equation of the Riemann $\zeta$-function, then we have equivalently
\begin{equation}\label{Gdelta}
  G_\delta(s) \ \ = \ \ \left\{
                          \begin{array}{ll}
                            \frac{\Gamma_{\mathbb R}(s)}{\Gamma_{\mathbb R}(1-s)}\,, & \delta\,\in\,2\Z, \\
                            i\, \frac{\Gamma_{\mathbb R}(s+1)}{\Gamma_{\mathbb R}(2-s)}\,, &\delta\,\in\,2\Z+1\,.
                          \end{array}
                        \right.
\end{equation}
Define
\begin{equation}\label{eq:gamma2}
\gama_+(s) \ \ = \ \  \prod\limits_{j=1}^N G_{\delta_j}(s+\lambda_j)^{-1}
\;\text{ and }\; \gama_-(s) \ \ = \ \  \prod\limits_{j=1}^N G_{1+\delta_j}(s+\lambda_j)^{-1}.
\end{equation}
The ratio of Gamma factors $\gama_\pm(s)$ appears naturally in the functional equations of the standard $L$-function of $F$ and its twists by Dirichlet characters.
  For this reason we can alternatively write
\begin{equation}
\gama_+(s) \ \ = \ \ \frac{L(1-s,\widetilde\pi)\,\epsilon(s,\pi)}{L(s,\pi)}
\quad\text{ and }\quad \gama_-(s) \ \ = \ \ \frac{L(1-s,\widetilde\pi\otimes \sign)\,\epsilon(s,\pi\otimes \sign)}{L(s,\pi\otimes \sign)}\,,\label{gamma}
\end{equation}
where the local factors are as defined in \cite[Appendix]{jacquet}.

Let $\omega\in C_c^\infty(\mathbb R_{>0})$
 and let $\tilde{\omega}(s)$ denote its Mellin transform.  We shall now clarify the relationship between $\omega$ and its Voronoi transform $\Omega$ from  (\ref{formalvoronoitransform}).  Decompose $\Omega$ into its even and odd parts for $y>0$
\begin{equation}\label{eq:Omega+-def_new}
\aligned
\Omega_+(y) \ \ & = \ \ \textstyle{\frac{1}{2}}\bra{ \Omega(y)+\Omega(-y) }\\
\Omega_-(y) \ \ & = \ \  \textstyle{\frac{1}{2}}\bra{ \Omega(y)-\Omega(-y) }.
\endaligned
\end{equation}
It then follows from \cite[(1.5)]{millerschmid2} that
\begin{equation}\label{eq:transform_integration_w_gamma}
\Omega_\pm	(x) \ \  = \ \  \frac{1}{2\pi i} \int_{\Re(s)=-\sigma}\tilde{\omega}(s) \, x^{s-1} \,  \gama_\pm(s)\dd\!s
\end{equation}
for $x>0$ and some $\sigma>0$. Please note that we take some $\sigma>0$ to avoid the poles of $\gama_\pm(s)$, which are on some right half plane. Also, $\Omega$ is defined over $\mathbb{R}\backslash \{0\}$ and $\Omega_\pm$ over $\mathbb{R}_{>0}$.

The original Voronoi formula  for $\GL(N)$, $N\ge 3$, in \cite{millerschmid1,millerschmid2}  was proven using automorphic distributions.  The methods of \cite[\S4]{millerschmid2} can be used to derive Theorem~\ref{thm:mainthm_testfunction} as well.  We shall however prove it using a reformulation in terms of Dirichlet series, which we state in Theorem~\ref{thm:balancedvoronoi}.  This reformulation will itself be proved by taking a finite average of a similarly-reformulated version of the $\GL(N)$ Voronoi summation formula in terms of Dirichlet series, which  can be found in \cite{kiralzhou}  (and \cite[(1.12)]{millerschmid2}\footnote{The formula stated here corrects a misprint propagating from  \cite[(1.9)]{millerschmid2}, where the first two arguments in the definition of the Kloosterman sum  were mistakenly switched. }):

\begin{theorem}[Voronoi formula on GL($N$) of Miller-Schmid \cite{millerschmid2}]\label{thm:originalvoronoi}
Let $F$ be a cuspidal automorphic form on $\ssl{N}{Z}\backslash \ssl N R$
with abelian Fourier coefficients $A(*,\ldots, *)$.  Assume that $F$ generates an irreducible representation $\pi$ of $\ssl N R$ and let   $\gama_\pm$ be the ratio of  Gamma  factors from \eqref{eq:gamma2}-\eqref{gamma}. Let $c>0$ be an integer and let $a$ be any integer with $(a,c)=1$. Denote by $\overline{a}$ the multiplicative inverse of $a$ modulo $c$.
Let $\vec{q}=(q_1,q_2,\cdots,q_{N-2})$ be an $(N-2)$-tuple of positive integers.
Then the additively-twisted Dirichlet series
\begin{equation}\label{eq:LqTwistedAdditively}
\mathcal{L}_{\vec{q}}( s , F , \bar a/c) =q_1^{(N-2)s}q_2^{(N-3)s}\cdots q_{N-2}^{s} \sum_{n=1}^\infty \frac{A(q_{N-2},\cdots, q_1,n)}{n^s}\,         
e\bra{\frac{\overline{a}n}{c}}	 ,
\end{equation}
which is initially convergent for $\Re{s}\gg 1$,
	 has an analytic continuation to an entire function of $s \in \C$  satisfying the functional
	equation
	\begin{equation}\label{eq:glnVoronoi}
	\gathered
		\mathcal{L}_{\vec{q}}(s,F, \bar a/c) \ \ = \qquad\qquad\qquad\qquad\qquad\qquad\qquad
\qquad\qquad\qquad\qquad\qquad\qquad\qquad\qquad\qquad\qquad \\
  \tfrac{\gama_+(s) - \gama_-(s)}{2} \sum_{d_1|q_1c}\sum_{d_2 | \frac{q_1q_2c}{d_1}}
		\cdots \sum_{d_{N-2}| \frac{q_1\ldots q_{N-2} c}{d_1 \ldots d_{N-3}}}
		 \sum_{n=1}^\infty
		\frac{A(n,d_{N-2}, \ldots, d_2, d_1) \Kl_{N-2}(a,n, c; \vec{q}, \vec{d}) }
		{n^{1-s} c^{Ns-1} d_1^{1-(N-1)s}d_2^{1-(N-2)s}\cdots d_{N-2}^{1-2s}}
		\\
		 +  \ \tfrac{\gama_+(s) + \gama_-(s)}{2}
\sum_{d_1|q_1c}\sum_{d_2 | \frac{q_1q_2c}{d_1}}
		\cdots \sum_{d_{N-2}| \frac{q_1\ldots q_{N-2} c}{d_1 \ldots d_{N-3}}}		
 \sum_{n=1}^\infty
		\frac{A(n,d_{N-2}, \ldots, d_2, d_1) \Kl_{N-2}(a, -n, c; \vec{q}, \vec{d}) }
		{n^{1-s} c^{Ns-1} d_1^{1-(N-1)s}d_2^{1-(N-2)s}\cdots d_{N-2}^{1-2s}}\,,	
		\endgathered
	\end{equation}
where $ \vec{d}=(d_1,\ldots,d_{N-2})$
(both terms on the righthand side converge for $\Re{s}\ll -1$ and have entire continuations to $s\in \C$).
\end{theorem}

\section{Proof}\label{sec:proof}

We  begin  by restating Theorem~\ref{thm:mainthm_testfunction} in the language of Dirichlet series, analogously to Theorem~\ref{thm:originalvoronoi}.
\begin{theorem}\label{thm:balancedvoronoi}
Let $F$ be a cuspidal automorphic form on $\ssl{N}{Z}\backslash \ssl N R$, $N\ge 3$,
with abelian Fourier coefficients $A(*,\ldots, *)$.  Assume that $F$ generates an irreducible representation $\pi$ of $\ssl N R$ and let  $\gama_\pm$ be the ratio of  Gamma  factors from \eqref{eq:gamma2}-\eqref{gamma}.
Let $L$ and $M$ be two non-negative integers whose sum $L+M=N-2$.
Let $c>0$ be an integer and let $a$ be any integer with $(a,c)=1$.
Denote by $\overline{a}$ the multiplicative inverse of $a$ modulo $c$.
Let $\vec{q}=(q_1,q_2,\cdots,q_L)$ be an $L$-tuple of positive integers and
$\vec{Q}=(Q_1,Q_2,\cdots,Q_M)$ an $M$-tuple of positive integers.
Define the Dirichlet series
\begin{equation}\label{eq:bal_voronoi_left_in_original_theorem}
L_{\vec{q}, \vec{Q}}( s , F ,  \bar a/c ) =\sumD \sum_{n=1}^\infty \frac{A(q_{L},\cdots, q_1,D_1,\cdots, D_M ,n)\Kl_M(\bar{a},n, c; \vec{Q}, \vec{D})}{n^s}\frac{  q_1^{Ls} \cdots	q_L^{s}}
{ D_1^{(M+1)s-M} \cdots D_M^{2s-1}},
\end{equation}
where $\sumD$ stands for $ \sum_{D_1|Q_1c}\sum_{D_2 | \frac{Q_1Q_2c}{D_1}}
		\cdots \sum_{D_{M}| \frac{Q_1\ldots Q_M c}{D_1 \ldots D_{M-1}}}$.
	This Dirichlet series is  convergent for $\Re{s}\gg 1$, and  has an analytic continuation to an entire function in  $s \in \C$ which satisfies the functional
	equation
	\begin{equation}\label{eq:blaVoronoi}
\gathered
		L_{\vec{q},\vec{Q}}(s,F,  \bar a/c) \ \  = \qquad\qquad\qquad\qquad\qquad\qquad\qquad\qquad\qquad\qquad\qquad\qquad\qquad\qquad\\
c^{M+1-Ns}\left[  \tfrac{\gama_+(s) + (-1)^{M+1} \gama_-(s)}{2} L_{\vec{Q}, \vec{q}}(1-s,\widetilde{F},a/c)+\tfrac{\gama_+(s) + (-1)^{M} \gama_-(s)}{2} L_{\vec{Q}, \vec{q}}(1-s,\widetilde{F},-a/c)\right],\\
  \endgathered
\end{equation}
where $L_{\vec{Q},\vec{q}}(..,\widetilde{F},...)$ is defined using the contragredient coefficients $\widetilde{A}(m_1,\ldots,m_{n-1})=A(m_{n-1},\ldots,m_1)$.
\end{theorem}

The following two lemmas are used in the proof of Theorem~\ref{thm:balancedvoronoi}.
		
\begin{lemma}\label{lemma:cancel}
 Let $C, Q$, and $b$ be positive integers and $y, a$ integers.
Assuming $b|QC$, $(y, QC/b)=1$ and $(a,C)=1$,
we have
\begin{equation}\label{eqn:lem32}
\sum_{D|QC}\;\sums_{x\,(\mod \tfrac{QC}{D})}e\bra{\frac{Dxa}{C}+	\frac{byx}{\tfrac{QC}{D}}} \ \ = \ \
\begin{cases}
QC, \quad &\text{ if }b=Q \text{ and } y\equiv -a \;(\mod C ),\\
0, \quad &\text{ otherwise.}
\end{cases}
\end{equation}
\end{lemma}

\begin{proof}
The sum $\sum\limits_{z \,(\mod QC)}e\bra{\frac{z(Qa+by)}{QC}}$ equals $QC$ when $QC|Qa+by$, and vanishes otherwise.  Factoring each $z$ as $z=Dx$ with $D=\gcd(z,QC)$ and $x\in(\Z/\frac{QC}{D}\Z)^*$, we see this sum equals the lefthand side of (\ref{eqn:lem32}).
It thus suffices to show that the nonvanishing conditions are  equivalent.  Clearly
 $QC|Qa+by$ if  $b=Q$ and $y\equiv -a \;(\mod C )$.  Conversely, suppose $QC|Qa+by$.  Thus $Q|by$, which implies that   $Q$ divides $\gcd(by,QC)=b$; also,
  since we have assumed that $b|QC$, we must have $b|Qa$ and hence $b$ divides $\gcd(Qa,QC)=Q$.  Being divisors of each other, $b$ and $Q$ are equal; this forces $C|(a+y)$.
\end{proof}

\begin{proof}[Proof of Theorem \ref{thm:balancedvoronoi}]
We   open up the hyper-Kloosterman sums on the lefthand side  of \eqref{eq:bal_voronoi_left_in_original_theorem}
completely, which results in the formal identity
\begin{align}
&\sumD \sum_{n=1}^\infty \frac{A(q_{L},\cdots, q_1,D_1,\cdots, D_M ,n)\Kl_M(\bar{a},n, c; \vec{Q}, \vec{D})}{n^s}\frac{  q_1^{Ls} \cdots	q_L^{s}}
{ D_1^{(M+1)s-M} \cdots D_M^{2s-1}}
\label{eq:balanced_voronoi_left}\\
= \ \ &\sumD \sum_{n=1}^\infty \frac{A(q_{L},\cdots, q_1,D_1,\cdots, D_M ,n)
}{n^s}
\frac{  q_1^{Ls} \cdots	q_L^{s}}
{ D_1^{(M+1)s-M} \cdots D_M^{2s-1}}\nonumber\\
& \! \! \sums_{x_1(\mod \frac{Q_1c}{D_1})}\
\sums_{x_2(\mod \frac{Q_1 Q_2 c}{D_1 D_2})}\!
\cdots\!
\sums_{x_M(\mod \frac{Q_1\cdots Q_M c}{D_1 \cdots D_M})}\!\!\!\!
e\bra{\frac{D_1x_1\bar{a}}{c} + \frac{D_2x_2 \overline{x_1}}{\frac{Q_1c}{D_1}}
		+ \cdots +
		\frac{D_{M}x_{M} \overline{x_{M-1}}}{\frac{Q_1\cdots Q_{M-1} c}{D_1 \cdots D_{M-1} }}
		+ \frac{n \overline{x_{M}}}{\frac{Q_1\cdots Q_{M} c}{D_1 \cdots D_{M} }}}\nonumber\\
= \ \ &
\sumD
\frac{1}{ D_1^{Ns-M} \cdots D_M^{Ns-1}}
\sums_{x_1(\mod \frac{Q_1c}{D_1})}
\!\cdots\!\!
\sums_{x_M(\mod \frac{Q_1\cdots Q_M c}{D_1 \cdots D_M})}
e\bra{\frac{D_1x_1\bar{a}}{c} + \frac{D_2x_2 \overline{x_1}}{\frac{Q_1c}{D_1}}
		+ \cdots +
		\frac{D_{M}x_{M} \overline{x_{M-1}}}{\frac{Q_1\cdots Q_{M-1} c}{D_1 \cdots D_{M-1} }}}
\nonumber\\
&\quad \;\quad \;\quad \;\quad \;\quad   \times \ \sum_{n=1}^\infty \frac{A(q_{L},\cdots, q_1,D_1,\cdots, D_M ,n)
}{n^s}
{{ D_1^{(L+1)s}\cdots  D_M^{(L+M)s} }     q_1^{Ls} \cdots	q_L^{s} }
		 e\bra{\frac{n \overline{x_{M}}}{\frac{Q_1\cdots Q_{M} c}{D_1 \cdots D_{M} }}}.\nonumber
\end{align}
By Theorem \ref{thm:originalvoronoi}, the $n$-sum part is absolutely convergent for $\Re s\gg 1$ and has analytic continuation to $\mathbb{C}$, hence the same assertions are true of \eqref{eq:bal_voronoi_left_in_original_theorem}.
Applying (\ref{eq:glnVoronoi}) to the $n$-sum, we get
\begin{align*}
 &
\sumD
 D_1^{M-1} D_2^{M-2}\cdots D_{M-1}
\sums_{x_1(\mod \frac{Q_1c}{D_1})}
\!\cdots\!\!
\sums_{x_M(\mod \frac{Q_1\cdots Q_M c}{D_1 \cdots D_M})}
e\bra{\frac{D_1x_1\bar{a}}{c} + \frac{D_2x_2 \overline{x_1}}{\frac{Q_1c}{D_1}}
		+ \cdots +
		\frac{D_{M}x_{M} \overline{x_{M-1}}}{\frac{Q_1\cdots Q_{M-1} c}{D_1 \cdots D_{M-1} }}}
		\\
&\sum_{b_M|D_M\cc}\cdots\sum_{b_1|\frac{D_M\cdots D_1}{b_{M}\cdots b_2}\cc}\sum_{d_1|q_1\cb}\cdots \sum_{d_L|\frac{q_1\cdots q_L}{d_1\cdots d_{L-1}}\cb}\\
&\quad\;  \ \ \sum_{n=1}^\infty\frac{A(n, d_L,\cdots,d_1,b_1,\cdots,b_M)}{n^{1-s}
\bra{{Q_1\cdots Q_{M} c}}^{Ns-1}b_M^{1-(L+M+1)s}\cdots b_1^{1-(L+2)s} d_1^{1-(L+1)s}\cdots d_L^{1-2s}
		 }\\
		 &
\quad\;\quad\; \times\!\left[\tfrac{\gama_+(s)-\gama_-(s)}{2}\Kl_{N-2}\bra{ x_M, n,\cc;(D_M,\cdots,D_1,q_1,\cdots,q_L), (b_M,\cdots,b_1,d_1,\cdots,d_L)}
\right.
\\
&\quad\;\quad \ \   \left. +\tfrac{\gama_+(s)+\gama_-(s)}{2}
\Kl_{N-2}\bra{ x_M,- n,\cc;(D_M,\cdots,D_1,q_1,\cdots,q_L), (b_M,\cdots,b_1,d_1,\cdots,d_L)}
\right]\!,
\end{align*}		
which is absolutely convergent for $\Re s\ll -1$.
We open up the hyper-Kloosterman sum partially, obtaining
\begin{align*}
	 &\Kl_{N-2}\bra{ x_M, n,\cc;(D_M,\cdots,D_1,q_1,\cdots,q_L), (b_M,\cdots,b_1,d_1,\cdots,d_L)}\\
	=&\sums_{y_M(\mod \tfrac{D_M}{b_M}\cc)} \ \sums_{y_{M-1}(\mod \tfrac{D_M D_{M-1}}{b_M b_{M-1}}\cc)}  \cdots\sums_{y_1(\mod \tfrac{D_M\cdots D_1}{b_M\cdots b_1}\cc)}
		\\
&
\quad \;\quad \;
e\bra{\frac{b_M y_M x_M}{\cc} + \frac{b_{M-1} y_{M-1} \overline{y_M}}{\frac{D_M}{b_M}\cc}+ \cdots +  \frac{b_1y_1\bar{y}_2}{\tfrac{D_M\cdots D_2}{b_M\cdots b_2}\cc} }
\Kl_L\bra{\bar{y}_1,n,\cb;(q_1,\cdots,q_L),(d_1,\cdots,d_L)}.
\end{align*}
After reordering the summations, $L_{\vec{q},\vec{Q}}(s,F,  \bar a/c) $ equals \addtocounter{equation}{1}
\begin{align*}
 & \nonumber
\sum_{D_1|Q_1c}
\ \sums_{x_1(\mod \frac{Q_1c}{D_1})}
\cdots  \sum_{D_{M-1}| \frac{Q_1\ldots Q_{M-1} c}{D_1 \ldots D_{M-2}}}
	\	\sums_{x_{ M-1}(\mod \frac{Q_1\cdots Q_{M-1} c}{D_1 \cdots D_{M-1}})}
\sum_{D_{M}| \frac{Q_1\ldots Q_M c}{D_1 \ldots D_{M-1}}}
	\	\sums_{x_M(\mod \frac{Q_1\cdots Q_M c}{D_1 \cdots D_M})}
	\label{eq:sum_of_D}\tag{\theequation a}\\
	&\resizebox{0.922\hsize}{!}{$
	\sideset{}{}\sum\limits_{b_M|\frac{Q_1\cdots Q_Mc}{D_1\cdots D_{M-1}}}
\			\sums\limits_{y_M(\mod \frac{Q_1\cdots Q_Mc		}{D_1\cdots D_{M-1}b_M})}
\sideset{}{}\sum\limits_{b_{M-1}|\frac{Q_1\cdots Q_Mc}{D_1\cdots D_{M-2}b_M}}
\			\sums\limits_{y_{M-1}(\mod \frac{Q_1\cdots Q_Mc		}{D_1\cdots D_{M-2}b_{M-1}b_M})}
\cdots
\sideset{}{}\sum\limits_{b_1|\frac{Q_1\cdots Q_M c}{b_2\cdots b_M}}
\ \sums\limits_{y_1(\mod \frac{Q_1\cdots Q_M c	}{b_1\cdots b_M})}
$}
\label{eq:sum_of_b}\tag{\theequation b}\\
%
%
&\quad\;
 \sum_{d_1|q_1\cb}\cdots \sum_{d_L|\frac{q_1\cdots q_L}{d_1\cdots d_{L-1}}\cb} D_1^{M-1} D_2^{M-2}\cdots D_{M-1}\label{eq:sum_of_d}\tag{\theequation c}
\\
&\quad\;\quad\quad\;e\bra{\tfrac{D_1x_1\bar{a}}{c} + \tfrac{D_2x_2 \overline{x_1}}{\tfrac{Q_1c}{D_1}}
		+ \cdots +
		\tfrac{D_{M}x_{M} \overline{x_{M-1}}}{\tfrac{Q_1\cdots Q_{M-1} c}{D_1 \cdots D_{M-1} }}\;+\;\tfrac{b_M y_M x_M}{\cc} + \tfrac{b_{M-1} y_{M-1} \overline{y_M}}{\tfrac{D_M}{b_M}\cc}+ \cdots +  \tfrac{b_1y_1\bar{y}_2}{\tfrac{D_M\cdots D_2}{b_M\cdots b_2}\cc} }	\nonumber\\
&\quad\;\quad\;\quad\;\quad
\sum_{n=1}^\infty\frac{A(n, d_L,\cdots,d_1,b_1,\cdots,b_M)}{n^{1-s}
\bra{{Q_1\cdots Q_{M} c}}^{Ns-1}b_M^{1-(L+M+1)s}\cdots b_1^{1-(L+2)s} d_1^{1-(L+1)s}\cdots d_L^{1-2s}
		 }	\nonumber
		\\
		 &
\quad\;\quad\;\quad\;\quad\;\quad\;\times\left[\tfrac{\gama_+(s)-\gama_-(s)}{2}
\Kl_L\bra{\bar{y}_1,n,\cb;(q_1,\cdots,q_L),(d_1,\cdots,d_L)}
\right.	\nonumber
\\
&\quad\;\quad\;\quad\;\quad\;\quad\;\quad \left. +\tfrac{\gama_+(s)+\gama_-(s)}{2}
\Kl_L\bra{\bar{y}_1,-n,\cb;(q_1,\cdots,q_L),(d_1,\cdots,d_L)}
\right].	\nonumber
\end{align*}
Observe that $D_M$ is not present in the summations in lines \eqref{eq:sum_of_b} and \eqref{eq:sum_of_d}.
 Thus consider the $D_M$- and $x_M$-summations,
$$\sum_{D_{M}| \frac{Q_1\ldots Q_M c}{D_1 \ldots D_{M-1}}}\ \sums_{x_M(\mod \frac{Q_1\cdots Q_M c}{D_1 \cdots D_M})}e\bra{		\frac{D_{M}x_{M} \overline{x_{M-1}}}{\frac{Q_1\cdots Q_{M-1} c}{D_1 \cdots D_{M-1} }}+\frac{b_M y_M x_M}{\cc} }, $$
to which we apply Lemma \ref{lemma:cancel} with $C=\frac{Q_1\cdots Q_{M-1}c}{D_1\cdots D_{M-1}}$, $Q=Q_M$, $a=\overline{x_{M-1}}$, $b=b_M$, $x=x_M$ and $y=y_M$.
This forces $b_M=Q_M$ and $y_M\equiv -\overline{x_{M-1}}\; (\mod \frac{Q_1\cdots Q_{M-1}c}{D_1\cdots D_{M-1}})$, for  otherwise, the quantity is zero.  Note that indeed the modulus of $x_{M-1}$,
 $${\frac{Q_1\cdots Q_{M-1} c}{D_1 \cdots D_{M-1} }} \ \ = \ \ {\frac{Q_1\cdots Q_{M} c}{D_1 \cdots D_{M-1}b_M }}\,,
$$
is precisely the modulus  of $y_M$.  The overall expression is multiplied by $QC=\frac{Q_1\cdots Q_Mc}{D_1\cdots D_{M-1}}$, so the portion of the summand in  (\ref{eq:sum_of_d}) becomes $D_1^{M-2}D_2^{M-3}\cdots D_{M-2}$.

Now that $b_M$ and $y_{M}$ have been removed, we see that the remaining indices of summation  in \eqref{eq:sum_of_b} and \eqref{eq:sum_of_d} do not involve $D_{M-1}$.   We continue to apply
  Lemma \ref{lemma:cancel} to $D_{M-1}$- and $x_{M-1}$-summations,
$$\sum_{D_{M-1}| \frac{Q_1\ldots Q_{M-1} c}{D_1 \ldots D_{M-2}}}\ \sums_{x_{M-1}(\mod \frac{Q_1\cdots Q_{M-1} c}{D_1 \cdots D_{M-1}})}e\bra{		\frac{D_{M-1}x_{M-1} \overline{x_{M-2}}}{\frac{Q_1\cdots Q_{M-2} c}{D_1 \cdots D_{M-2} }}+\frac{-b_{M-1} y_{M-1} x_{M-1}}{\tfrac{Q_1\cdots Q_{M-1} c}{D_1 \cdots D_{M-1} }} } .$$
This forces $b_{M-1}=Q_{M-1}$ and $y_{M-1}\equiv (-1)^2\overline{x_{M-2}}\; (\mod \frac{Q_1\cdots Q_{M-2}c}{D_1\cdots D_{M-2}})$. In turn, we consecutively apply  Lemma \ref{lemma:cancel} to $D_j$- and $x_j$-summations for $j=M-2,\ldots,2$,
$$\sum_{D_{j}| \frac{Q_1\ldots Q_{j} c}{D_1 \ldots D_{j-1}}}\ \sums_{x_{j}(\mod \frac{Q_1\cdots Q_{j} c}{D_1 \cdots D_{j}})}e\bra{		\frac{D_{j}x_{j} \overline{x_{j-1}}}{\frac{Q_1\cdots Q_{j-1} c}{D_1 \cdots D_{j-1} }}+\frac{(-1)^{M-j}b_{j} y_{j} x_{j}}{\tfrac{Q_1\cdots Q_{j} c}{D_1 \cdots D_{j} }} },$$
forcing $b_{j}=Q_{j}$ and $y_{j}\equiv (-1)^{M-j+1}\overline{x_{j-1}}\; (\mod \frac{Q_1\cdots Q_{j-1}c}{D_1\cdots D_{j-1}}).$
 At the final stage, we apply Lemma \ref{lemma:cancel} once more to the remaining sum
$$\sum_{D_1|Q_1c}\ \sums_{x_1(\mod \frac{Q_1c}{D_1})} e\bra{\frac{D_1x_1\bar a}{c}+\frac{(-1)^{M-1}b_1y_1x_1}{\tfrac{Q_1c}{D_1}}}
$$
to force $b_1=Q_1$ and  $y_1\equiv (-1)^M \bar a$ $(\mod c)$.
At this point the summations in \eqref{eq:sum_of_D} and \eqref{eq:sum_of_b} have all disappeared, and we find
(\ref{eq:balanced_voronoi_left}) is equal to
\begin{align*}
\sum_{d_1|q_1c}\cdots  \sum_{d_L|\frac{q_1\cdots q_L}{d_1\cdots d_{L-1}}c}\sum_{n=1}^\infty &\frac{A(n, d_L,\cdots,d_1,Q_1,\cdots,Q_M)\,c^MQ_1^M\cdots  Q_M}{n^{1-s}
\bra{{Q_1\cdots Q_{M} c}}^{Ns-1}Q_M^{1-(L+M+1)s}\cdots Q_1^{1-(L+2)s} d_1^{1-(L+1)s}\cdots d_L^{1-2s}
		 }
\\
&\times \left[\tfrac{\gama_+(s)-\gama_-(s)}{2}\Kl_L\bra{ (-1)^Ma, n,c;(q_1,\cdots,q_L), (d_1,\cdots,d_L)}
\right.
\\
&\quad \left. +\tfrac{\gama_+(s)+\gama_-(s)}{2}
\Kl_L\bra{  (-1)^{M}a, -n,c;(q_1,\cdots,q_L), (d_1,\cdots,d_L)}
\right]\\
 = \ \
 \tfrac{\gama_+(s) + (-1)^{M+1} \gama_-(s)}{2} & 
 \sumdd \sum_{n=1}^\infty
		\frac{A(n,d_{L}, \ldots,  d_1, Q_1,\cdots, Q_M) \Kl_L(a,n, c; \vec{q}, \vec{d}) }
		{n^{1-s} c^{Ns-1-M}}
\frac{Q_1^{M(1-s)}\cdots Q_M^{1-s}}{d_1^{1-(L+1)s}\cdots d_L^{1-2s}}
		\\
		+ \ \tfrac{\gama_+(s) + (-1)^M \gama_-(s)}{2} & \sumdd
\sum_{n=1}^\infty
		\frac{A(n,d_{L}, \ldots,  d_1, Q_1,\cdots, Q_M) \Kl_L(a,-n, c; \vec{q}, \vec{d}) }
		{n^{1-s} c^{Ns-1-M}}
\frac{Q_1^{M(1-s)}\cdots Q_M^{1-s}}{d_1^{1-(L+1)s}\cdots d_L^{1-2s}},
\end{align*}		
which is  equivalent  to (\ref{eq:blaVoronoi}).
		\end{proof}
%
%
%
%
%
%
%
%
%

\section*{Acknowledgment}
The authors would like to thank the two anonymous referees for their very helpful comments and suggestions.

\nothing{
\begin{theorem}[Li and Miller's balanced Voronoi formula on GL(4)]
\label{thm:limiller}
Let $F$ be a cuspidal automorphic form for $\mathrm{GL}(4,\mathbb{Z})$
with abelian Fourier coefficients $A(\;,\;,\;)$.
Let $\omega\in C_c^\infty(0,\infty) $  be a test function and \blue{$\Omega_+$} its \blue{transform defined in Appendix \ref{app:transforms}.} \cyan{\textit{Voronoi transform}\emph{ seems to be a non-conventional name. Should it be Hankel transform or something like that?}}
Then

\begin{align*}
&\sum_{m=-\infty }^{\infty }\sum_{d|N} d A(a,d,m)S\left( m,a;
\frac{N}{d}\right) \omega \bra{\frac{md^{2}}{(N,b)}}
\\ &=  \sum_{m=-\infty }^{\infty }
\sum_{d\left|\tfrac{aN}{(N,b)}\right.} \frac{d(N,b)^2}{N^2}A(m,d,(N,b))S\left( m,(N,b);\frac{aN}{(N,b)d}\right)  \blue{\Omega_+}\bra{\frac{nd^2(N,b)^4}{aN^4}}
\,.
\end{align*}

\end{theorem}
}

\;

\noindent {\scshape{Stephen D. Miller}} \\
Department of Mathematics\\
The State University of New Jersey\\
Piscataway, NJ 08854, USA\\
 miller@math.rutgers.edu
\\

\noindent {\scshape{Fan Zhou}} \\
Department of Mathematics and Statistics\\
The University of Maine\\
Orono, ME 04469, USA\\
fan.zhou@maine.edu

\end{document}